\documentclass[12pt, german]{article}

\makeatletter
\newfont{\footsc}{cmcsc10 at 8truept}
\newfont{\footbf}{cmbx10 at 8truept}
\newfont{\footrm}{cmr10 at 10truept}
\makeatother
\usepackage{amsmath}
\usepackage{amssymb}
\usepackage{exscale}
\usepackage{amsthm}
\usepackage{epsfig}
\usepackage{setspace}
\usepackage{pgfplots}
\usepackage{tikz}
\usepackage{tikz-cd}
\usetikzlibrary{quotes,angles,arrows,positioning}
\usetikzlibrary{arrows,shapes,positioning}
\usetikzlibrary{calc,decorations.markings}
\usepackage{verbatim}
\usepackage{caption}
\usepackage{ulem}
\usepackage{dsfont}

\theoremstyle{plain}
\newtheorem{theorem}{Theorem}[section]

\newtheorem{lemma}[theorem]{Lemma}
\newtheorem{corollary}[theorem]{Corollary}

\theoremstyle{definition}
\newtheorem{definition}[theorem]{Definition}

\newtheorem{remark}[theorem]{Remark}

\DeclareMathOperator{\supp}{supp}

\newcommand{\N}{{\mathbb{N}}}

\def\bsi{\boldsymbol{i}}
\def\bsj{\boldsymbol{j}}
\def\bs0{\bf 0}

\newcommand{\hp}[1]{\textcolor{red}{#1}}

\newtheorem{innercustomgeneric}{\customgenericname}
\providecommand{\customgenericname}{}
\newcommand{\newcustomtheorem}[2]{%
	\newenvironment{#1}[1]
	{%
		\renewcommand\customgenericname{#2}%
		\renewcommand\theinnercustomgeneric{##1}%
		\innercustomgeneric
	}
	{\endinnercustomgeneric}
}

\newcustomtheorem{customthm}{Theorem}
\newcustomtheorem{customcor}{Corollary}

\title{Improved bounds for the bracketing number of orthants or revisiting an algorithm of Thi\'{e}mard to compute bounds for the star discrepancy}

\author{Michael Gnewuch\thanks{Institut f\"ur Mathematik, Universit\"at Osnabr\"uck, 
Germany ({\tt michael.gnewuch@uni-osnabrueck.de}).}
}
\begin{document}

\maketitle
\vskip 1pc

\begin{abstract}
We improve the best known upper bound for the bracketing number of $d$-dimensional axis-parallel boxes anchored in $0$
(or, put differently, of lower left orthants intersected with the $d$-dimensional unit cube $[0,1]^d$). 
More precisely, we provide a better 
estimate
%upper bound 
for the cardinality of an algorithmic  bracketing cover construction due to Eric Thi\'emard, which 
%serves as a subroutine in 
forms the core of his algorithm to approximate the star discrepancy of arbitrary point sets from
[E. Thi\'emard, An algorithm to compute bounds for the star discrepancy, J.~Complexity 17 (2001), 850 -- 880]. 

Moreover, the new upper bound for the bracketing number of anchored axis-parallel boxes yields an improved upper 
estimate
%bound 
for the bracketing number of arbitrary axis-parallel boxes in $[0,1]^d$. 

In our upper bounds all constants are fully explicit. 
\end{abstract}

{\bf Keywords:}  
entropy number, 
covering number,
packing number, 
delta cover, 
discrepancy, 
%weighted star-discrepancy.
%negative correlation, Monte Carlo point sets,  
tractability.

\section{Introduction} 

Entropy numbers, like the logarithm of bracketing or covering numbers, are measures of the size and complexity of a given class of functions $\mathcal{F}$, and are frequently used in empirical process theory, density estimation, high-dimensional probability, machine learning, uniform distribution theory, or, Banach space theory, 
see, e.g., \cite{DL01, LT91, Sen22, vdVW96, Ver18}.

In this paper we focus on the bracketing number for characteristic functions of axis-parallel rectangles in the $d$-dimensional unit cube. Nevertheless, we start with the general definition of the bracketing number and put it into relation to other entropy concepts like covering and packing numbers or delta covers.

%\vspace{1ex}
%\framebox{DEF. BRACKETING NUMBER}
%\vspace{1ex}

Let us assume that $\mathcal{F}$ is a subset of a normed space $(F, \|\cdot \|)$ of real-valued functions. For two functions $\ell, u \in F$ with $\ell \le u$ we define the \emph{bracket} 
$[\ell, u ]$ by 
\begin{equation}\label{bracket}
[\ell, u ] := \{ f\in F \,|\, \ell \le f \le u \},
\end{equation}
where inequalities between functions are meant pointwise. The bracket is an \emph{$\varepsilon$-bracket} if its \emph{weight} $W([\ell, u])$ satisfies
\begin{equation}\label{general_weight}
W([\ell, u]) := \| u - \ell \| \le \varepsilon.
\end{equation}
A set of $\varepsilon$-brackets $\mathcal{B}$ is an \emph{$\varepsilon$-bracketing cover} of $\mathcal{F}$ if $\mathcal{F} \subseteq \cup_{B\in \mathcal{B}} B$. 
The \emph{bracketing number $N_{[\,]}(\varepsilon, \mathcal{F}, \|\cdot \|)$} is the minimum number of $\varepsilon$-brackets needed to cover $\mathcal{F}$. 

Let us state a related definition:
A finite subset $\Gamma$ of $F$ is called a (one-sided) $\delta$-cover of $\mathcal{F}$, if for every $h\in \mathcal{F}$ there exists a $\delta$-bracket $[\ell, u]$ with $\ell, u \in \Gamma$ and $h\in [\ell,u]$. Denote by $N(\delta, \mathcal{F}, \|\cdot\|)$ the minimum cardinality of all $\delta$-covers of $\mathcal{F}$.

The main difference between the notion of an $\varepsilon$-bracketing cover $\mathcal{B}$ and a $\delta$-cover $\Gamma$ (for $\delta = \varepsilon$) is that $\mathcal{B}$ consists of $\varepsilon$-brackets, i.e., subsets of $F$, while 
$\Gamma$ consists of functions, i.e., points of $F$. 
It is easy to verify the following relation:
\begin{equation}\label{relation_bracketing_delta_cover}
N(\varepsilon, \mathcal{F}, \|\cdot\|) 
\le 2 N_{[\,]}(\varepsilon, \mathcal{F}, \|\cdot \|)
\le N(\varepsilon, \mathcal{F}, \|\cdot\|) \left( N(\varepsilon, \mathcal{F}, \|\cdot\|) + 1\right).
\end{equation}
%The bounds in 
%\eqref{relation_bracketing_delta_cover} and 
%\eqref{relation_bracketing_covering_packing} 
The first inequality shows that upper bounds for the bracketing number give us reasonable upper bounds 
%for the covering and the packing number, and 
for $N(\varepsilon, \mathcal{F}, \|\cdot\|)$, while the second inequality 
%(but the other way around this does, in general, not yield 
seems to be, in general, not overly helpful to provide good upper bounds for the bracketing number with the help of upper bounds for $N(\varepsilon, \mathcal{F}, \|\cdot\|)$.
 
Techniques based on bracketing were introduced in the context of empirical process theory by 
R.~M.~Dudley in 1978 \cite{Dud78}. The notion of a (one-sided) $\delta$-cover was introduced in \cite{Mha04} in the context of approximation theory and was used frequently in the area of uniform distribution theory and quasi-Monte Carlo methods, see, e.g., the survey article \cite{Gne11} and the more recent research papers
\cite{Ais13, AW13,  AH14, DR14, DRZ16, Rud18,  HKKR20, WGH20, GH21, GPW21, FGW23}.

%\framebox{DEF. COVERING/PACKING NUMBER}
%\vspace{1ex}

Similarly, we can define the \emph{covering number $C(\varepsilon, \mathcal{F}, \|\cdot \|)$} as the minimum number of closed balls $\{ f \,|\, \|g-f\| \le \varepsilon\}$ in $F$, with some center $g\in F$ and radius $\varepsilon$, needed to cover $\mathcal{F}$. Furthermore, we call a subset $\mathcal{S}$ of $F$ \emph{$\varepsilon$-separated}, if $\| f - g \| > \varepsilon$ for all $f,g \in \mathcal{S}$ with 
$f \neq g$. 
The \emph{packing number $P(\varepsilon, \mathcal{F}, \| \cdot \| )$} is the maximum number of $\varepsilon$-separated functions in $\mathcal{F}$.

%\framebox{VERGLEICH BRACKETING, COVERING, PACKING NUMBER}
%\vspace{1ex}
Let us assume that $(F, \|\cdot\|)$ satisfies the Riesz property, i.e., for all $f,g \in F$ the pointwise inequality $|f| \le |g|$ implies $\|f\| \le \|g\|$. 
Then it is straightforward to show that
\begin{equation}\label{relation_bracketing_covering_packing}
C(\varepsilon, \mathcal{F}, \| \cdot \|) 
\le P(\varepsilon, \mathcal{F}, \| \cdot \| ) 
\le C(\varepsilon/2, \mathcal{F}, \| \cdot \|) 
\le N_{[\,]}(\varepsilon, \mathcal{F}, \| \cdot \|),
\end{equation}
where the Riesz property is used to establish the final inequality, cf., e.g., \cite[Chapter~2]{vdVW96}. 
%\vspace{1ex}
The inequalities in 
%\eqref{relation_bracketing_delta_cover} and 
\eqref{relation_bracketing_covering_packing} show that upper bounds for the bracketing number imply upper bounds for the covering and the packing number, 
%and for $N(\varepsilon, \mathcal{F}, \|\cdot\|)$, 
but the other way around this is not necessarily the case. 

Let us denote the $d$-dimensional Lebesgue measure by $\lambda^d$. We confine ourselves to the setting where $(F, \|\cdot \|)$ is the space 
$L^1(I^d, \lambda^d)$ of $\lambda^d$-integrable real-valued functions on $I^d := [0,1)^d$ 
and $\mathcal{F}$ is a set of characteristic functions of half-open boxes in the $d$-dimensional unit cube $I^d$. Note that $L^1(I^d, \lambda^d)$
satisfies the Riesz property. 

More precisely, for $x=(x_1, \ldots, x_d)$ and $y=(y_1, \ldots, y_d)$ in $I^d$  put
\[
[x,y) := [x_1, y_1) \times \cdots \times [x_d, y_d),
\]
and let 
\begin{equation}\label{corners_rectangles}
\mathcal{C}_d := \{ [0, x) \,|\, x \in I^d\}
\hspace{3ex}\text{and}\hspace{3ex}
\mathcal{R}_d := \{ [y, z) \,|\, y, z \in I^d\}
\end{equation}
be the set systems of all anchored half-open $d$-dimensional intervals (``corners'')  and all half-open intervals (``rectangles'') in $I^d$, respectively. 
%We are interested in the case where $F$ is given by
We identify these set systems with the set of characteristic functions
\begin{equation}\label{charistic_fcts_corners_rectangles}
{\bf 1}_{\mathcal{C}_d} := \{ 1_{C} \,|\, C \in \mathcal{C}_d \}
\hspace{3ex}\text{and}\hspace{3ex}
{\bf 1}_{\mathcal{R}_d} := \{ 1_{R} \,|\, R \in \mathcal{R}_d \},
\end{equation}
respectively; as a general convention, we denote the characteristic function of a set $A$ by $1_A$.

\begin{comment}
%DER NAECHSTE ABSATZ STIMMT SO NICHT!
Note first that in the case $\mathcal{F}  \in \{ {\bf 1}_{\mathcal{C}_d} ,{\bf 1}_{\mathcal{R}_d} \}$ 
%and $\mathcal{F} = L^1(I^d, \lambda^d)$  
for constructing explicit $\varepsilon$-bracketing or $\delta$-covers
%and for determining the corresponding bracketing and covering numbers, 
we may confine ourselves to brackets $[\ell, u]$ whose lower left and upper right functions are in $F$. Indeed, we may always replace $\ell$ by $1_L$ and $u$ by $1_U$,
where 
\begin{equation*}
L := \bigcap_{\supp( \ell) \subseteq C \in \mathcal{C}_d} C
\hspace{3ex}\text{and}\hspace{3ex}
U:= \bigcup_{C \subseteq u^{-1}((-\infty, 1])} C
\end{equation*}
if $F= {\bf 1}_{\mathcal{C}_d}$, and  can handle it in a corresponding fashion if $F= {\bf 1}_{\mathcal{R}_d}$.
\end{comment}

For the set systems $\mathcal{C}_d$ and $\mathcal{R}_d$ we would like to have small $\varepsilon$-bracketing covers which can be efficiently constructed. Furthermore, we want good upper bounds for the bracketing numbers 
$N_{[\,]}(\varepsilon, \mathcal{C}_d, \|\cdot \|_{L^1})$ and 
$N_{[\,]}(\varepsilon, \mathcal{R}_d, \|\cdot \|_{L^1})$, which reveal the dependence of these quantities on $\varepsilon$ \emph{and} on $d$. 

The  so far best known upper bound with explicitly given constants for the bracketing number of $\mathcal{C}_d$ is
\begin{equation}\label{bound_GPW}
N_{[\,]}(\varepsilon, \mathcal{C}_d, \|\cdot \|_{L^1})
\le \max (1, 1.1^{d-101}) \frac{d^d}{d!} (\varepsilon^{-1} + 1)^d.
\end{equation} 
It was proved in \cite[Theorem~2.5]{GPW21} and relies on a construction from \cite{Gne08a}; due to the  derivation of a new Faulhaber-type inequality and a refined analysis it improves on the bound presented in 
\cite[Theorem~1.15]{Gne08a}.

Furthermore, \cite[Lemma~1.18]{Gne08a} provides a constructive way %of how 
to obtain an $\varepsilon$-bracketing cover for $\mathcal{R}_d$ with the help of an  $\tfrac{\varepsilon}{2}$-bracketing cover for $\mathcal{C}_d$. In particular, it establishes the relation
\begin{equation}\label{rel_C_R}
N_{[\,]}(\varepsilon, \mathcal{R}_d, \|\cdot \|_{L^1}) 
\le \left( N_{[\,]} \left( \tfrac{\varepsilon}{2}, \mathcal{C}_d, \|\cdot \|_{L^1} \right) \right)^2.
\end{equation}
Combining the estimates \eqref{bound_GPW} and \eqref{rel_C_R} results in the 
so far best known upper bound for the bracketing number of $\mathcal{R}_d$.

%%% BEGINN ALTE VERSION %%%%%%%
\begin{comment}
The best known upper bounds for the bracketing numbers so far are
\begin{equation}\label{bound_GPW}
N_{[\,]}(\varepsilon, \mathcal{C}_d, \|\cdot \|_{L^1})
\le \max (1, 1.1^{d-101}) \frac{d^d}{d!} (\varepsilon^{-1} + 1)^d,
\end{equation} 
and 
\begin{equation}\label{bound_extreme}
N_{[\,]}(\varepsilon, \mathcal{R}_d, \|\cdot \|_{L^1})
\le 
 \max (1, 1.1^{2(d-101)}) \frac{d^{2d}}{(d!)^2} (\varepsilon^{-1} + 1)^{2d}.
\end{equation} 
The first bound was proved in \cite[Theorem~2.5]{GPW21} and relies on a construction from \cite{Gne08a}; due to the  derivation of a new Faulhaber-type inequality and a refined analysis it improves on the bound presented in 
\cite[Theorem~1.15]{Gne08a}.
The bound \eqref{bound_extreme} follows from \eqref{bound_GPW} in combination with 
\cite[Lemma~1.18]{Gne08a}, and is again constructive. %cf. also \cite{FGW23}.
In particular, \cite[Lemma~1.18]{Gne08a} provides the relation 
\end{comment}
%%% ENDE ALTE VERSION %%%%%%%

In \cite{Thi01a} Eric Thi\'emard gave a simple procedure to partition the unit cube $I^d$, which results in an $\varepsilon$-bracketing cover $\mathcal{P}^d_\varepsilon$ for $\mathcal{C}_d$. This procedure is the key ingredient in his algorithm to compute bounds for the star discrepancy of arbitrary point sets in $I^d$
(for more information about approaches to calculate discrepancy measures we refer to the book chapter \cite{DGW14}).
He also provided an upper bound for the cardinality of his bracketing cover.
Let us focus here in the introduction on the most interesting case where $d\ge 3$. 
(We provide in Theorem~\ref{main_thm} and Remark~\ref{Rem:Improvements} also  results for $d=2$, and discuss the previously known results in two dimensions in Remark~\ref{Rem:Case_d_2}.) 
Thi\'emard's bound reads (after some slight simplification to make it more easily comparable to the other bounds presented in this note)
\begin{equation}\label{bound_thiemard}
| \mathcal{P}^d_\varepsilon | \le \frac{d^d}{d!} \left( \frac{\ln \left( \varepsilon^{-1} \right)}{\varepsilon} + 1 \right)^d
\hspace{3ex}\text{for $d\ge 3$.}
\end{equation}
In this note we make a more refined analysis of Thi\'emard's partitioning procedure %which yields $\mathcal{P}^d_\varepsilon$ 
to derive a substantially better upper bound than \eqref{bound_thiemard}. In particular, we show in Theorem~\ref{main_thm} that 
\begin{equation}\label{improv_bound_d_ge_3}
| \mathcal{P}^d_\varepsilon | \le \frac{d^d}{d!} \left( \frac{1}{\varepsilon} \right)^d
\hspace{3ex}\text{for $d\ge 3$,}
\end{equation}
and provide in Theorem~\ref{main_thm} and Remark~\ref{Rem:Improvements} %additional bounds for the case $d=2$ and 
some further (moderate) improvements. %for arbitrary $d$. 
Note that these bounds improve reasonably on the bound \eqref{bound_thiemard} and yield additionally better upper bounds for the bracketing numbers of $\mathcal{C}_d$ and, via \eqref{rel_C_R}, of $\mathcal{R}_d$. These are good news, since the construction of Thi\'emard is relatively simple and can be implemented  easily, cf. \cite[Algorithm~3]{Thi01a}.  In contrast, the construction from \cite{Gne08a} that leads to the bound \eqref{bound_GPW} 
from \cite{GPW21} is rather complicated. 
The key ingredients for the improvements in Theorem~\ref{main_thm} and Remark~\ref{Rem:Improvements} are the new Lemmas~\ref{Lemma:Representation_delta} 
and \ref{Lemma:Monotonie}.

The newly derived bounds on $N_{[\,]}(\varepsilon, \mathcal{C}_d, \|\cdot \|_{L^1})$ and $N_{[\,]}(\varepsilon, \mathcal{C}_d, \|\cdot \|_{L^1})$ can be used to improve
other bounds as, e.g., the tractability bounds for the usual as well as the weighted star discrepancy and extreme discrepancy in \cite{Gne08a, GPW21, FGW23}.
We confine ourselves here to compare our newly derived upper bound on the $L^1$-packing number $P(\varepsilon, \mathcal{C}_d, \|\cdot \|_{L^1})$ for the set system $\mathcal{C}_d$ with the celebrated upper bound of David Haussler \cite{Hau95} (which is actually applicable to general set systems with finite  Vapnik-Chervonenkis (VC) dimension):
Since the set system $\mathcal{C}_d$ has VC dimension $d$,
Haussler's bound  reads 
\[
P(\varepsilon, \mathcal{C}_d, \|\cdot \|_{L^1}) \le (d+1) 2^d e^{d+1} \varepsilon^{-d},
\]
see \cite[Corollary~1]{Hau95}, while \eqref{improv_bound_d_ge_3} and \eqref{relation_bracketing_covering_packing} yield for $d\ge 3$
\[
P(\varepsilon, \mathcal{C}_d, \|\cdot \|_{L^1}) \le \frac{d^d}{d!} \, \varepsilon^{-d} \le  \frac{1}{\sqrt{2 \pi d}}\, e^{d}\, \varepsilon^{-d}\,;
\]
for the second estimate cf. \cite{Robbins55}.

%\hp{\framebox{OUR RESULTS AND THEIR VALUE!!!}}

\section{Revisiting an Algorithm of Eric Thi\'{e}mard}

%\framebox{DESCRIPTION PARTITIONING PROCESS}
We now describe Thi\'emard's algorithmic partitioning process of the $d$-dimensional unit cube that results for given $\varepsilon \in (0,1)$ and $d\in\N$ 
after a finite number of steps in an $\varepsilon$-bracketing cover  $\mathcal{P}^d_\varepsilon$ of $\mathcal{C}_d$. 

To simplify matters, we identify each anchored half-open $d$-dimensional interval $[0,x)$ in $\mathcal{C}_d$ with its upper right corner point $x\in [0,1]^d$. Since we already agreed to identify the sets $[0,x)$ in $\mathcal{C}_d$ with their corresponding characteristic functions $1_{[0,x)}$, cf. \eqref{corners_rectangles} and
\eqref{charistic_fcts_corners_rectangles}, this finally results in identifying a point $x\in [0,1]^d$ with the corresponding function $1_{[0,x)}$. In this sense, we can identify for two given points $x,y \in [0,1]^d$ satisfying $x \le y$ (meant component-wise) the intersection of the bracket $[1_{[0,x)}, 1_{[0,y)}]$ (cf. \eqref{bracket})
%(or, more precisely, that bracket 
and ${\bf 1}_{\mathcal{C}_d}$ with the closed $d$-dimensional interval $[x,y]$, which we also call a bracket. Consequently, we define its weight to be  
$$
W([x,y]) := W( [1_{[0,x)}, 1_{[0,y)}] ).
$$ 
Recall that the latter weight 
is given by
\begin{equation*}
W( [1_{[0,x)}, 1_{[0,y)}] ) = \| 1_{[0,y)} - 1_{[0,x)}] \|_{L^1} = \lambda^d([0,y)) - \lambda^d([0,x)),
\end{equation*}
cf. \eqref{general_weight}.
The partitioning process should result in a partition of the $d$-dimensional unit cube; to achieve this it is necessary to work with half-open intervals instead of closed ones. To avoid picky distinctions, we will identify half-open $d$-dimensional intervals $[x,y)$ with their corresponding bracket $[x,y]$
%, cf. \eqref{bracket}, 
and, consequently, call $[x,y)$ itself a bracket.
Accordingly, the weight of such a bracket is given by 
\begin{equation}\label{weight}
W([x,y)) 
%:= W([x,y]) 
:= \lambda^d([0,y)) - \lambda^d([0,x)) = \prod^d_{j=1} y_j - \prod^d_{j=1} x_j.
\end{equation}
%cf. \eqref{general_weight}.
During the partitioning process the $d$-dimensional unit cube $I^d = [0,1)^d$ is recursively decomposed into $d$-dimensional (half-open) subintervals until all those subintervals are 
$\varepsilon$-brackets, i.e., have a weight at most $\varepsilon$. 
More precisely, it starts with $I^d$, which is a subinterval of type $1$, and partitions it into $d$ subintervals $Q^{I^d}_1, \ldots, Q^{I^d}_d$, where $Q^{I^d}_j$ is a subinterval of type $j$, and a %$\varepsilon$-bracket 
subinterval  
$Q^{I^d}_{d+1}$ of type $d+1$, which has weight exactly $\varepsilon$; $Q^{I^d}_{d+1}$  is added to the bracketing cover $\mathcal{P}^d_\varepsilon$. 
If for $j\in \{1, \ldots, d\}$ the subinterval $Q^{I^d}_j$ has weight at most $\varepsilon$, it will be added to $\mathcal{P}^d_\varepsilon$; otherwise it will later be partitioned into smaller subintervals.
More generally, in the recursion step, a subinterval 
\[
P= [\alpha, \beta)
\]
of type $j \in \{1,\ldots, d\}$ with $\alpha = \alpha^P, \beta = \beta^P \in [0,1]^d$ and $W(P) > \varepsilon$ 
is partitioned into subintervals
\begin{equation*}
\begin{split}
Q_1^P &= [(\alpha_1, \alpha_2, \ldots, \alpha_d), (\gamma_1, \beta_2, \ldots, \beta_d)),\\
Q_2^P &= [(\gamma_1, \alpha_2, \ldots, \alpha_d), (\beta_1, \gamma_2, \beta_3, \ldots, \beta_d)),\\
\vdots &\hspace{8ex} \vdots \\
Q_d^P &= [(\gamma_1,  \ldots, \gamma_{d-1}, \alpha_d), (\beta_1,  \ldots, \beta_{d-1}, \gamma_d)),\\
Q_{d+1}^P &= [\gamma, \beta),
\end{split}
\end{equation*}
where the vector $\gamma = \gamma^P \in [0,1]^d$ is given as in Theorem~\ref{Thi01a:Theorem3.2}.
In particular, if $j>1$ then we have $Q^P_1 = \cdots = Q^P_{j-1} = \emptyset$ and
all these empty subintervals will not be considered further  (and will, in particular, not be elements of the final bracketing cover). The remaining subintervals $Q^P_k$, $k=j, j+1, \ldots, d$, are subintervals of type $k$, and $Q_{d+1}^P$ is an $\varepsilon$-bracket with $W(Q_{d+1}^P) = \varepsilon$,  which will be added to $\mathcal{P}^d_\varepsilon$.
Thi\'emard called the subroutine of his algorithm that corresponds to this recursion step {\sc decompose}$(P,j)$; we will use the same name to refer to the recursion step above.

%If we already have $W(P) \le \varepsilon$, then $P$ will be added to the bracketing cover $\mathcal{P}^d_\varepsilon$ and not be decomposed any further.
If for $k\in \{j, j+1, \ldots, d\}$ the subinterval $Q^P_k$ has weight larger $\varepsilon$, then it will be decomposed further; otherwise it will be added to $\mathcal{P}^d_\varepsilon$.

For the convenience of the reader we restate two major results from \cite{Thi01a} concerning the partitioning process of the unit cube $I^d$, namely \cite[Theorem~3.2]{Thi01a} and \cite[Corollary~3.2]{Thi01a}. Note that formula (9) in \cite[Theorem~3.2]{Thi01a} contains a typo -- the power appearing there  is incorrect. We corrected it in the corresponding formula \eqref{eq:delta_thiemard} below.

\begin{theorem}[\cite{Thi01a}]
\label{Thi01a:Theorem3.2}
For each call of {\sc decompose}$(P,j)$ during the decomposition process of $I^d$,
where $P = [\alpha^P, \beta^P)$, 
%and each $i\in \{j, \ldots,s\}$ 
we have
\begin{equation}\label{eq:gamma}
\gamma_i^P = 
\begin{cases}
\, \alpha_i^P 
\hspace{2ex}&\text{if $i < j$,}\\
\, \delta^P \beta_i^P
\hspace{2ex} &\text{if $i \ge j$,}
\end{cases}
\end{equation}
where
\begin{equation}\label{eq:delta_thiemard}
\delta^{P} = \left( \frac{\prod_{i=1}^d \beta_i^P - \varepsilon}{\prod_{i=1}^{j-1} \alpha_i^P \prod_{i=j}^d \beta_i^P} \right)^{\frac{1}{d-j+1}}.
\end{equation}
\end{theorem}

\begin{corollary}[\cite{Thi01a}]
\label{Thi01a:Corollary3.2}
For each call of {\sc decompose}$(P,j)$ during the decomposition process of $I^d$, the interval $P= [\alpha^P, \beta^P)$ is partitioned into exactly $d-j+2$ subintervals, the $d-j+1$ subintervals $Q^P_j, \ldots, Q^P_d$ of common weight $\delta^P W(P)$, and the subinterval $Q_{d+1}^P = [\gamma^P, \beta^P)$ of weight $\varepsilon$.
\end{corollary}

In the following remark we collect some helpful observations.

\begin{remark}\label{Rem:**}
During the decomposition process the $\alpha$-component of index $d$ is preserved for all subintervals that are not of type $d+1$: 
\begin{equation}\label{eq:alphas_0}
\alpha^{Q_i^P}_d = \alpha^P_d = \cdots = \alpha^{I^d}_d = 0,
\end{equation}
cf. \cite[Proof of Lemma~3.2]{Thi01a}. Hence we can infer for all subintervals $P$ 
appearing in the decomposition process and being of type $j < d+1$ that
\begin{equation}\label{eq:weight_P_only_betas_appear}
W(P) = \prod_{i=1}^d \beta_i^{P} - \prod_{i=1}^d \alpha_i^{P}  
= \prod_{i=1}^d \beta_i^{P}.
\end{equation}
Furthermore, we have  $Q_j^P = [ \alpha^{Q_j^P}, \beta^{Q_j^P} )$ with
\begin{equation}\label{eq:alpha_beta_gamma}
\alpha_i^{Q_j^P} = 
\begin{cases}
\, \gamma_i^P 
\hspace{2ex}&\text{if $i < j$},\\
\, \alpha_i^P
\hspace{2ex} &\text{if $i \ge j$,}
\end{cases}
\hspace{5ex}
\beta_i^{Q_j^P} = 
\begin{cases}
\, \gamma_i^P 
\hspace{2ex}&\text{if $i = j$},\\
\, \beta_i^P
\hspace{2ex} &\text{if $i \neq j$,}
\end{cases}
\end{equation}
see \cite[identity (4)]{Thi01a}.
Let $P$ be a subinterval of type $j$. Then we get from
\eqref{eq:weight_P_only_betas_appear}
and Corollary~\ref{Thi01a:Corollary3.2}
for each $i\in \{j, \ldots, d\}$ the identity
\begin{equation}\label{eq:weight_Q_i_P}
W(Q_i^{P}) = \prod_{\ell=1}^d \beta_\ell^{Q_i^P} =  \delta^{P} W(P).
\end{equation}
We always have $\delta^P \in (0,1)$, cf. \cite[Section~3.2]{Thi01a}.  
\end{remark}

To be able to improve Thi\'emard's upper bound for the cardinality of 
$\mathcal{P}^d_\varepsilon$, we need a suitable recursion formula that shows how the parameter $\delta=\delta^P$ evolves in the course of the partitioning process.
The desired formula will be proved in the next lemma.

\begin{lemma}\label{Lemma:Representation_delta}
For each call of {\sc decompose}$(P,j)$ during the decomposition process of $I^d$
where $\delta^{P} W(P) > \varepsilon$,
%\begin{equation}\label{eq:weight_Q_i_P}
%W(Q_i^{P}) = \prod_{\ell=1}^d \beta_\ell^{Q_i^P} =  \delta^{P} W(P),
%\end{equation}
we obtain for every $i\in \{j, \ldots, d\}$  in  {\sc decompose}$(Q^P_i,i)$
\begin{equation}\label{eq:representation_delta}
\delta^{Q_i^{P}} = \left( \frac{W(P) - \varepsilon/\delta^P}{W(P) - \varepsilon} \right)^{\frac{1}{d-i+1}}\delta^{P}.
\end{equation}
The quantity $\delta^{Q_i^{P}}$ is strictly monotone increasing with respect to the parameters $\delta^P$ and $W(P)$ and strictly monotone decreasing with respect to the parameter $i$.
\end{lemma}

\begin{proof}
%The identities in \eqref{eq:weight_Q_i_P} follow immediately from
%\eqref{eq:weight_P_only_betas_appear}
%and Corollary~\ref{Thi01a:Corollary3.2}.
%Furthermore, 
We obtain from \eqref{eq:alpha_beta_gamma} 
and \eqref{eq:gamma}
%Theorem~\ref{Thi01a:Theorem3.2} 
\[
\prod_{\ell =1}^{i-1} \alpha_\ell^{Q_i^P} \prod_{\ell=i}^d \beta_\ell^{Q_i^P}  
= \prod_{\ell =1}^{i} \gamma_\ell^{P} \prod_{\ell=i+1}^d \beta_\ell^{P}  
= (\delta^P)^{i-d} \prod_{\ell =1}^{d} \gamma_\ell^{P},
\]
and from \eqref{weight}, \eqref{eq:weight_P_only_betas_appear}, and Corollary~\ref{Thi01a:Corollary3.2}
\[
\prod_{\ell =1}^{d} \gamma_\ell^{P} 
= \prod_{i=1}^d \beta_i^{P} - W([\gamma^P, \beta^P)) 
= W(P) - \varepsilon.
\]
Plugging these findings and identity \eqref{eq:weight_Q_i_P} into equation \eqref{eq:delta_thiemard} yields 
identity \eqref{eq:representation_delta}. The monotonicity statements are now easily verified.
\end{proof}

Let us introduce some additional notation:
We put 
\[
Q^{(0)}_0 := I^d 
\hspace{3ex} \text{and}\hspace{3ex}
Q^{(1)}_{j_1} := Q^{Q^{(0)}_0 }_{j_1}
\hspace{3ex} \text{for $j_1 = 1,\ldots, d$}.
\]
For $r\in\N$ and $\bsj \in S^r$, where
\[
S^r := \{ \bsi \in \N^r\, |\, 1 \le i_1 \le i_2 \le \ldots \le i_r \le d\},
\] 
we define for an element $Q^{(r)}_{\bsj}$ of the decomposition process with
$W(Q^{(r)}_{\bsj}) > \varepsilon$
\[
Q^{(r+1)}_{(\bsj, j_{r+1})} :=  Q^{Q^{(r)}_{\bsj} }_{j_{r+1}}
\hspace{3ex} \text{for $j_{r+1} = j_r,\ldots, d$}.
\]
By convention, we put $S^0:= \{0\}$. Furthermore, we denote for $r\in\N$ the vector $(1, \ldots, 1) \in \N^r$ by ${\bf 1}_r$. 

\begin{definition}\label{height_partition}
For given $d\in\N$ and $\varepsilon \in (0,1)$ let %us define 
the \emph{height of the partition $\mathcal{P}^d_\varepsilon$}, denoted by $h = h(d, \varepsilon)$, be the largest number in $\N$ such that in the partitioning process there exists
a $\bsj \in S^{h-1}$ with $W(Q^{(h-1)}_{\bsj}) > \varepsilon$. 
\end{definition}

%%%%%%%%%%%%%%%%%%%%%%%%%%%%%%%%%%%%%%%
%    ABOUT TRIANGULAR TREES
%%%%%%%%%%%%%%%%%%%%%%%%%%%%%%%%%%%%%%%

%\framebox{\hp{UEBERLEITUNG!!!}}

%Before we 
%describe the actual algorithm of Thi\'{e}mard, 
To prove his upper bound for the cardinality of $\mathcal{P}^d_\varepsilon$, Thi\'{e}mard represented the elements of $\mathcal{P}^d_\varepsilon$ as vertices in a certain type of planar graph, namely a triangular tree. 
We restate the definition and the (for us) most relevant result for triangular trees from \cite[Section~3.4]{Thi01a}:

\begin{definition}
Let $w \in\N$ and $h \in \N_0$. The \emph{triangular tree} $T^w_h$ of \emph{height} $h$ and \emph{width} $w$ is recursively defined:
\begin{itemize}
\item $T^w_0$ is a leaf.
\item If $h\in \N$, then $T^w_h$ is a tree made up of a root with $w$ attached triangular trees $T^w_{h-1}, \ldots, T^1_{h-1}$.
\end{itemize}
\end{definition}

\begin{theorem}[\cite{Thi01a}]\label{Thm:Leaves}
Let $w \in\N$ and $h \in \N_0$. The \emph{triangular tree} $T^w_h$ has $N^w_h := \binom{h+w-1}{w-1}$ leaves.
\end{theorem}

A proof of the theorem can be found in \cite[Section~3.4]{Thi01a}.

%%% START: BRACKETING COVER REPRESENTED AS A TREE %%%
\begin{comment}
Simarly, we can easily deduce from the definition of the triangular tree the following result.

\begin{corollary}[\cite{Thi01a}]\label{Thm:Vertices}
Let $w \in\N$ and $h \in \N_0$. The \emph{triangular tree} $T^w_h$ has $N^{w+1}_h = \binom{h+w}{w}$ vertices.
\end{corollary}

\begin{proof}
Let $M^w_h$ denote the number of vertices of $T^w_h$. 
We prove the statement via iduction over the heigth $h$. 

For $h=0$ the tree $T^w_0$ is a leaf, i.e., it consists exactly of $1={ 0+ w \choose w}$ vertices.

Now let us assume that the claimed result holds for $h$. Then we have, due to the definition ``triangular tree'' and the well-known summation formula for binomial coefficients
\begin{equation*}
M^w_{h+1} 
= 1 +\sum_{\ell = 1}^w M^\ell_h 
= 1 + \sum_{\ell = 1}^w { h + \ell \choose \ell }
=  \sum_{\ell = 0}^w { h + \ell \choose \ell }
%\sum_{\ell = 0}^h N^w_{\ell} 
%= \sum_{\ell = 0}^h { \ell + w -1 \choose w-1}
%= \sum_{\ell = 0}^h { \ell + w -1 \choose \ell}
= { h + (w  + 1) \choose w}
= { (h + 1) + w  \choose w}.
\end{equation*}
\end{proof}
\end{comment}
%%% END: BRACKETING COVER REPRESENTED AS A TREE %%%

Now fix $d\in \N$ and $\varepsilon \in (0,1)$, and let $h=h(d,\varepsilon)$ be the height of the partition $\mathcal{P}^d_\varepsilon$.
%Next let us explain how for a suitable $h=h(d,\varepsilon)$ 
In the following we explain how Thi\'{e}mard related 
%why for 
%any dimension $d\in \N$ and any $\varepsilon \in (0,1]$ 
the bracketing cover $\mathcal{P}^d_\varepsilon$ %of can be
%identified for   
%related to 
to a subtree $S^{d+1}_h$ of $T^{d+1}_h$  having the same root such that there is a one-to-one correspondence between the elements of  $\mathcal{P}^d_\varepsilon$ and the leaves of
$S^{d+1}_h$. Indeed, we can represent all $d$-dimensional intervals that appear in the decomposition process of $I^d$ as vertices in a tree. Its root is the $d$-dimensional unit cube $I^d$ itself, and if $P$ is an interval of type $j$ that is generated in the course of the decomposition process, then it is a vertex of the tree and two things may happen: 
If $W(P) \le \varepsilon$ (which is, in particular, the case if $j=d+1$), then $P$ is a leaf (which corresponds to the fact that $P$ will not be decomposed further); else
$P$ has $d-j+2$ children (which corresponds to the fact that $P$ will be decomposed via {\sc Decompose}$(P,j)$ into $d-j+2$ subintervals).
Since the decomposition process terminates after finitely many steps, the resulting tree $S^{d+1}_h$ is obviously a subtree of a triangular tree $T^{d+1}_h$
of height $h=h(d,\varepsilon)$. 
%Note that in this sense the height $h$ of the partition $\mathcal{P}^d_\varepsilon$ is the smallest height of a triangular tree $T^{d+1}_h$ which contains the subgraph

Clearly, the number of leaves in $S^{d+1}_h$ is at most the number of leaves in $T^{d+1}_h$. Hence, due to Theorem~\ref{Thm:Leaves}, this would result in 
\begin{equation}\label{Thi_est_Partition}
|\mathcal{P}^d_\varepsilon| \le \binom{h+ d}{d},
\end{equation}
see \cite[Theorem~3.4]{Thi01a}.

%%%%%%%%%%%%%%%%%%%%%%%%%%%%%%%%%%%%%%%

The next lemma is crucial for being able to prove a good upper bound on $h$. 

 \begin{lemma}\label{Lemma:Monotonie}
 Let $r\in \{1, \ldots, h\}$, and let $\bsi, \bsj \in S^r$ such that $i_\nu \le j_\nu$ 
 for $\nu =1, \ldots, r$. 
 Assume that $Q^{(r)}_{\bsi}$ and $Q^{(r)}_{\bsj}$ appear in the decomposition process of $I^d$. 
Then we have
 \begin{equation}\label{sharp_1}
 W(Q^{(r)}_{\bsi}) \ge W(Q^{(r)}_{\bsj}) 
  \end{equation}
% \hspace{3ex} \text{and, if $W(Q^{(r)}_{\bsi}) > \varepsilon$,}\hspace{3ex}
and, if $W(Q^{(r)}_{\bsj}) > \varepsilon$,
 \begin{equation}\label{sharp_2}
 \delta^{Q^{(r)}_{\bsi}} \ge  \delta^{Q^{(r)}_{\bsj}}.
 \end{equation}
Inequality \eqref{sharp_1} is strict if and only if for some $\nu \in \{ 1, \ldots, r-1\}$ we have $i_\nu < j_\nu$, while
inequality \eqref{sharp_2} is strict if and only if there exists some
 $\nu \in \{ 1, \ldots, r\}$ such that $i_\nu < j_\nu$.
 
 %In particular, the height $h=h(d, \varepsilon)$ of $\mathcal{P}^d_\varepsilon$ is 
 %given by 
 %\begin{equation*}
 %h = \min\{ r \in \N_0 \,|\, W(Q^{(r)}_{{\bf 1}_r}) \le \varepsilon \}.
 %\end{equation*}
 \end{lemma}
 
 \begin{proof}
 We prove the inequalities via induction on $r$. 
 
 For $r=1$ %Lemma~\ref{Lemma:Representation_delta} 
 Corollary~\ref{Thi01a:Corollary3.2} 
 implies $W(Q^{(1)}_1) = \cdots = W(Q^{(1)}_d) = \delta^{Q^{(0)}_0} W(Q^{(0)}_0)$. 
 Furthermore, 
 if $\delta^{Q^{(0)}_0} W(Q^{(0)}_0) > \varepsilon$, 
 then Lemma~\ref{Lemma:Representation_delta} yields
 $\delta^{Q^{(1)}_1} > \delta^{Q^{(1)}_2} > \cdots > \delta^{Q^{(1)}_d}$.
 
 Now let $r\in \{2,\ldots, h\}$. We assume that the statement of Lemma~\ref{Lemma:Monotonie} holds for $r-1$. 
 Note that  $Q^{(r)}_{\bsi}$ and $Q^{(r)}_{\bsj}$ 
 appear in the decomposition process of $I^d$ if and only if 
 $Q^{(r-1)}_{(i_1, \ldots,i_{r-1})}$ and $Q^{(r-1)}_{(j_1, \ldots,j_{r-1})}$
 appear in the decomposition process and have weights strictly larger than 
 $\varepsilon$. 
 In case that 
 $Q^{(r-1)}_{(i_1, \ldots,i_{r-1})} = Q^{(r-1)}_{(j_1, \ldots,j_{r-1})}$ the statement follows again from Corollary~\ref{Thi01a:Corollary3.2} and  Lemma~\ref{Lemma:Representation_delta}.
 Otherwise our induction hypothesis yields
 \begin{equation*}
 w:= W(Q^{(r-1)}_{(i_1, \ldots,i_{r-1})}) \ge W(Q^{(r-1)}_{(j_1, \ldots,j_{r-1})}) =: w'
 \end{equation*}
% \hspace{3ex}\text{and}\hspace{3ex}
and
\begin{equation*}
 \delta:= \delta^{Q^{(r-1)}_{(i_1, \ldots,i_{r-1})}} > \delta^{Q^{(r-1)}_{(j_1, \ldots,j_{r-1})}} =: \delta'.
 \end{equation*}
 Now %Lemma~\ref{Lemma:Monotonie} 
 identity~\eqref{eq:weight_Q_i_P} 
 implies
 \[
 W(Q^{(r)}_{\bsi}) = \delta w > \delta' w'=  W(Q^{(r)}_{\bsj}),
 \]
 and Lemma~\ref{Lemma:Representation_delta} yields 
 \[
 \delta^{Q^{(r)}_{\bsi}} = \left( \frac{w-\varepsilon/\delta}{w-\varepsilon} \right)^{\frac{1}{s-i_r+1}} \delta  
%\ge \left( \frac{w-\varepsilon/\delta}{w-\varepsilon} \right)^{\frac{1}{s-j_r+1}} \delta  
 > \left( \frac{w'-\varepsilon/\delta'}{w'-\varepsilon} \right)^{\frac{1}{s-j_r+1}} \delta'  
 =  \delta^{Q^{(r)}_{\bsj}}. 
 \]
 %The formula for $h$ follows now from \eqref{sharp_1} 
 %and Definition~\ref{height_partition}.
 \end{proof}

\begin{theorem}\label{main_thm}
Let $d\in \N$ and $\varepsilon \in (0,1]$. The height $h=h(d, \varepsilon)$ of the partition $\mathcal{P}^d_\varepsilon$ is given by
\begin{equation*}
 h = \min\{ r \in \N_0 \,|\, W(Q^{(r)}_{{\bf 1}_r}) \le \varepsilon \},
 \end{equation*}
%the smallest $r\in \N_0$ such that $W(Q^{(r)}_{{\bf 1}_r}) \le \varepsilon$, 
and it satisfies the upper bound
\begin{equation}\label{est_h}
 h \le \lceil d (\varepsilon^{-1} - 1)\rceil.
\end{equation}
The cardinality of the $\varepsilon$-bracketing cover $\mathcal{P}^d_\varepsilon$
can be estimated as
\begin{equation}\label{est_card_cover_d_2}
|\mathcal{P}^2_\varepsilon| \le 2 \left( \frac{1}{\varepsilon} + \frac{1}{2} \right) 
\frac{1}{\varepsilon}
\end{equation}
if $d=2$, and as
\begin{equation}\label{est_card_cover}
|\mathcal{P}^d_\varepsilon| 
\le  \frac{d^d}{d!} \left( \frac{1}{\varepsilon} - \frac{1}{2} \left( 1 - \frac{3}{d} \right) \right)^{d}
\le \frac{d^d}{d!} \varepsilon^{-d}
\end{equation}
if $d\ge 3$. 
%Inequality \eqref{est_card_cover}
%is still valid if we replace the factor $(\varepsilon^{-1} + 1)^d$ by $(\varepsilon^{-1} + 3/4)^d$ if $d=2$, by $(\varepsilon^{-1} + 1/d)^d$ for $d=3,4$, and $\varepsilon^{-d}$ for all $d\ge 5$. 
\end{theorem}

\begin{proof}
%STILL TO BE WRITTEN DOWN!
For $r\in \N$ we put $w_r := W(Q^{(r)}_{{\bf 1}_r})$
and $\delta_r:= \delta^{Q^{(r)}_{{\bf 1}_r}}$, and additionally we set $w_0 := W(Q^{(0)}_0) = 1$ and $\delta_0 := \delta^{Q^{(0)}_0} = (1-\varepsilon)^{1/d}$.  
Due to \eqref{eq:weight_P_only_betas_appear} and \eqref{eq:delta_thiemard} we obtain
\[
\delta_r = \left( \frac{w_r - \varepsilon}{w_r} \right)^{1/d}
\hspace{3ex}\text{for all $r\in \N_0$,}
\]
and due to \eqref{eq:weight_Q_i_P} 
%Corollary~\ref{Thi01a:Corollary3.2} 
we have $w_r = \delta_{r-1} w_{r-1}$
for all $r\in \N$. This yields
\begin{equation}\label{course_estimate_eps_over_s}
w_{r-1} - w_r  
= (1-\delta_{r-1}) w_{r-1} = \left( 1 - \left( 1 - \varepsilon/w_{r-1} \right)^{1/d}   \right) w_{r-1} 
\ge \varepsilon/d,
\end{equation}
where in the last step we used the inequality $(1-x)^{1/d} \le 1 - x/d$, which holds true for all $x\in [0,1]$.

%Recall that 
Due to Definition~\ref{height_partition} and Lemma~\ref{Lemma:Monotonie} 
the height $h$ of the partition $\mathcal{P}^d_\varepsilon$ is given by
\[
h = \min \{ r\in \N_0 \,: \, w_r \le \varepsilon \},
\]
and note that $w_r \le \varepsilon$ if and only if
\[
1-w_r = \sum^{r-1}_{\ell = 0} (w_\ell - w_{\ell + 1}) \ge 1-\varepsilon.
\]
Due to \eqref{course_estimate_eps_over_s} we have
\[ 
\sum^{r-1}_{\ell = 0} (w_\ell - w_{\ell + 1}) \ge r \cdot \varepsilon/d, 
\]
showing that for $w_r \le \varepsilon$ it suffices to have
\[
r \ge d(\varepsilon^{-1} - 1).
\]
Hence 
\[
h \le \left\lceil \frac{d}{\varepsilon} \right\rceil - d <  \frac{d}{\varepsilon} - d +1.
\]
Due to \eqref{Thi_est_Partition} we obtain 
\begin{equation}\label{est_card_cover_a_bit_better}
|\mathcal{P}^d_\varepsilon| 
\le  {d + h \choose d}
\le \frac{d^d}{d!} \left( \frac{1}{\varepsilon} + \frac{1}{d} \right) 
\frac{1}{\varepsilon} \left( \frac{1}{\varepsilon} - \frac{1}{d} \right) 
\cdots \left( \frac{1}{\varepsilon} - \frac{d-2}{d} \right).
%\le \frac{d^d}{d!} (\varepsilon^{-1} + 1)^d.
\end{equation}
This yields in the case $d = 2$ 
\begin{equation*}
|\mathcal{P}^2_\varepsilon| \le 2 \left( \frac{1}{\varepsilon} + \frac{1}{2} \right) 
\frac{1}{\varepsilon}.
\end{equation*}
In the case where $d \ge 3$, we employ the fact that a product consisting of $d$ positive factors is always at most as large as the $d$-th power of the arithmetic mean 
of the factors, which gives us
\begin{equation*}
|\mathcal{P}^d_\varepsilon| 
\le  \frac{d^d}{d!} \left( \frac{1}{\varepsilon} - \frac{1}{2} \left( 1 - \frac{3}{d} \right) \right)^{d}
\le  \frac{d^d}{d!} \varepsilon^{-d}.
\end{equation*}
\end{proof}

\begin{remark}[Some Further Improvements]\label{Rem:Improvements}
If we apply in \eqref{course_estimate_eps_over_s} 
a less coarse estimate, we may obtain a slightly better result than \eqref{est_h}.
Indeed, consider the function $f: [0,1) \to [1,\infty)$, given by
\begin{equation*}
f(x) := 1 + \sum^\infty_{k=1} \frac{(k-\tfrac{1}{d}) \cdot \ldots \cdot (1-\tfrac{1}{d})}{(k+1)!}\, x^k.
\end{equation*}
Using the Taylor series expansion of the function $(1-x)^{1/d}$, we get for $w_{\ell -1} > \varepsilon$ the identity
\begin{equation*}
w_{\ell -1} - w_\ell  
= \left( 1 - \left( 1 - \varepsilon/w_{\ell-1} \right)^{1/d}   \right) w_{\ell - 1} 
= \frac{\varepsilon}{d} \cdot f \left( \frac{\varepsilon}{w_{\ell -1}} \right).
%\ge \varepsilon/d,
\end{equation*}
Since $f$ is a monotonic increasing function and $w_{\ell -1} \le 1$, we have 
\[
w_{\ell -1} - w_\ell \ge \frac{\varepsilon}{d} \cdot f \left( \varepsilon \right),
\]
resulting in 
\[
1- w_r = \sum^{r-1}_{\ell = 0} (w_\ell - w_{\ell + 1})
\ge r \cdot  \frac{\varepsilon}{d} \cdot f \left( \varepsilon \right).
\]
Analogously as in the proof of Theorem~\ref{main_thm}, we obtain that the height $h$ satisfies
\begin{equation}\label{impro_h}
h = h(d, \varepsilon) \le \left\lceil \frac{d (\varepsilon^{-1} - 1)}{f(\varepsilon)} \right\rceil.
\end{equation}
For instance, employing the first-order truncation $1+ \tfrac{d-1}{2d} \varepsilon \le f(\varepsilon)$, \eqref{impro_h} implies the estimate
\begin{equation*}
h = h(d, \varepsilon) \le \left\lceil \frac{d (\varepsilon^{-1} - 1)}{1+ \tfrac{d-1}{2d} \varepsilon} \right\rceil, 
\end{equation*}
which for moderately small $\varepsilon$ clearly improves on \eqref{est_h}.

By using the same arguments as in the proof of Theorem~\ref{main_thm}, we now may improve \eqref{est_card_cover_d_2} and \eqref{est_card_cover}
by 
\begin{equation}\label{impro_1}
|\mathcal{P}^2_\varepsilon| \le 2 \left( \frac{\varepsilon^{-1}-1}{f(\varepsilon)} + \frac{3}{2} \right) 
\left( \frac{\varepsilon^{-1}-1}{f(\varepsilon)} + 1 \right) 
\end{equation}
if $d=2$, and by
\begin{equation}\label{impro_2}
|\mathcal{P}^d_\varepsilon| 
\le  \frac{d^d}{d!} \left(  \frac{\varepsilon^{-1}-1}{f(\varepsilon)} + \frac{1}{2} \left( 1 + \frac{3}{d} \right) \right)^{d}
\end{equation}
if $d\ge 3$. 
\end{remark}

\begin{remark}[The Case $d=2$]\label{Rem:Case_d_2}
The upper bounds for the bracketing number presented in Theorem~\ref{main_thm} 
and Remark~\ref{Rem:Improvements}
improve on the best known bounds so far in the case $d\ge 3$. 
In the case where $d=2$, actually better bounds are known. Indeed, in \cite[Proposition~5.1]{Gne08b} it was shown constructively, that the asymptotic behavior of the bracketing number in dimension $2$ is 
\begin{equation}\label{old_d_2_asymptotic}
N_{[\,]}(\varepsilon, \mathcal{C}_2, \|\cdot \|_{L^1}) \le \varepsilon^{-2} + o(\varepsilon^{-2});
\end{equation}
a matching lower bound was proved in \cite[Theorem~1.5]{Gne08a}. 

An upper bound with fully explicit constants was established constructively in 
\cite[Theorem~2.2]{GPW21} and reads 
\begin{equation}\label{old_d_2_pre_asymptotic}
N_{[\,]}(\varepsilon, \mathcal{C}_2, \|\cdot \|_{L^1}) 
\le 2 \ln(2) \varepsilon^{-2} + 3(\ln(2) + 1) \varepsilon^{-1} - \left( \tfrac{13}{9} \ln(2) - 1 \right);
\end{equation}
note that this bound has a smaller coefficient in front of the most important power $\varepsilon^{-2}$ than the bound \eqref{est_card_cover_d_2}.

The constructions that yield \eqref{old_d_2_asymptotic} and  \eqref{old_d_2_pre_asymptotic}, respectively, are different from Thi\'emard's construction.  But in 
\cite[Proposition~3.1]{Gne08b} the following asymptotic bound for Thi\'emard's bracketing cover was proved:
\begin{equation*}
|\mathcal{P}^2_\varepsilon| \le 2 \ln(2) \varepsilon^{-2} + O(\varepsilon^{-1}).
\end{equation*}
Moreover, numerical results in \cite{Gne08b} indicate that Thi\'emard's bracketing cover $\mathcal{P}^2_\varepsilon$ has a comparable or even slightly smaller cardinality than the one that yielded the bound \eqref{old_d_2_pre_asymptotic}, but a clearly larger one than the one that implied the bound \eqref{old_d_2_asymptotic}.
The problem with the latter construction is that it is not obvious how to generalize it to arbitrary dimensions in a way such that its cardinality can be upper bounded with reasonable effort.
\end{remark}

%%%%%%%%%%%%%%%%%%%%%%%%%%%%%%%%%%%%%%
\begin{comment}
\framebox{\hp{NOCH ZU ERGAENZEN:}}
\begin{itemize}
\item Asymptotic bound for Thi\'{e}mard's bracketing cover for $d=2$ from
\cite[Proposition~3.1]{Gne08b}, and additional bound with explicit constants
from \cite[Theorem~2.2]{GPW21}.
\item Improved asymptotic bound for bracketing number in the case where $d=2$ from \cite[Proposition~5.1]{Gne08b}.
\end{itemize}
\end{comment}
%%%%%%%%%%%%%%%%%%%%%%%%%%%%%%%%%%%%%%

\paragraph*{Acknowledgment} 
This note was initiated at the Dagstuhl workshop 23351, ``Algorithms and Complexity for Continuous Problems". The author thanks the Leibniz Center for Informatics, Schloss Dagstuhl, and its staff for their support and their hospitality. 

He also would like to thank two anonymous referees for their helpful comments.
%and their providing a stimulating workshop atmosphere.

{\footnotesize

\bibliographystyle{plain}
\bibliography{References}

}

\end{document}